\documentclass{amsart}
\usepackage{amsfonts,amssymb,amscd,amsmath,enumerate,verbatim,calc}
\usepackage{xy}

\newcommand{\wrt}{with respect to}

\newcommand{\R}{\mathcal{R} }
\newcommand{\Sc}{\mathcal{S}}
\newcommand{\F}{\mathcal{F}}
\newcommand{\Z}{\mathbb{Z} }

\newcommand{\rt}{\rightarrow}

\newcommand{\ov}{\overline}

\newcommand{\grade}{\operatorname{grade}}

\newcommand{\height}{\operatorname{height}}

\theoremstyle{plain}

\newtheorem{theorem}{Theorem}[section]

\theoremstyle{definition}

\newtheorem{s}[theorem]{}

\theoremstyle{remark}

\begin{document}

\title[Symbolic powers]{Symbolic powers  of monomial ideals}
\author{Tony~J.~Puthenpurakal}
\date{\today}
\address{Department of Mathematics, IIT Bombay, Powai, Mumbai 400 076}

\email{tputhen@math.iitb.ac.in}
\subjclass{Primary 13D40; Secondary 13H15}
\keywords{quasi-polynomials, monomial ideals, symbolic powers}
 \begin{abstract}
Let $A = K[X_1,\ldots, X_d]$ and let $I$, $J$ be monomial ideals in $A$. Let $I_n(J) = (I^n \colon J^\infty)$ be the $n^{th}$ symbolic power of $I$ \wrt \ $J$. It is easy to see that the function $f^I_J(n) = e_0(I_n(J)/I^n)$ is of quasi-polynomial type, say of period $g$ and degree $c$. For $n \gg 0$ say
\[
f^I_J(n) = a_c(n)n^c + a_{c-1}(n)n^{c-1} + \text{lower terms},
\] 
where for $i = 0, \ldots, c$,  $a_i \colon \mathbb{N} \rt \mathbb{Z}$ are periodic functions of period $g$ and $a_c \neq 0$.
 In \cite[ 2.4]{HPV} we (together with Herzog and Verma) proved that  $\dim I_n(J)/I^n$ is constant for $n \gg 0$ and $a_c(-)$ is a constant.  In this  paper we prove
 that if  $I$ is generated by some elements  of the same degree and $\height I \geq 2$ then
$a_{c-1}(-)$ is also a constant.
\end{abstract}
 \maketitle
\section{introduction}
Let us recall that a function $f \colon \mathbb{N} \rt \mathbb{Z}$ is called  of quasi-polynomial of period $g$ and degree $c$ if 
\[
f(n) = a_c(n)n^c + a_{c-1}(n)n^{c-1} + \text{lower terms}
\]
where for $i = 0, \ldots, c$,  $a_i \colon \mathbb{N} \rt \mathbb{Z}$ are periodic functions of period $g$ and $a_c \neq 0$.
 Many interesting functions are of quasi-polynomial type 
(i.e., it coincides with a quasi-polynomial for $n \gg 0$). For instance the Hilbert function of a graded module over a not-necessarily 
standard graded $K$-algebra is of quasi-polynomial type; see \cite[4.4.3]{BH}.

Let $f$ be of quasi-polynomial type. Following Erhart we  let the  \emph{grade} of  $f$ 
 denote the smallest integer $\delta \geq -1$  such that $a_j(-)$ is constant  for all $j > \delta$. Erhart had conjectured that if $P$ is a $d$-dimensional rational polytope   in $\mathbb{R}^m$ such that  for some $\delta$ the affine span of every $\delta$-dimensional face of $P$ contains a point with integer coordinates then  the grade of the Erhart quasi-polynomial of $P$ is $< \delta$. This conjecture was proved by  
 independently by McMullen  \cite{M} and Stanley \cite[Theorem 2.8]{S}. For a purely algebraic proof see \cite[Theorem 5]{BI}. 

\s\label{int} Our interest in quasi-polynomials arise in the following context (see \cite{HPV}): Let $A = K[X_1,\ldots, X_d]$ be a standard graded polynomial ring over a field $K$.
If $M$ is a graded $A$-module then let $e_0(M)$ denote its multiplicity.
 Let $I, J$ be monomial ideals. Let
$I_n(J) = I^n \colon J^\infty$ be the $n^{th}$ symbolic power of $I$ \wrt \ $J$. Then by \cite[3.2]{HHT} we have that $\bigoplus_{n \geq 0} I_n(J)$ is a finitely generated $A$-algebra. From this it is easy to see that $f^I_J(n) = e_0(I_n(J)/I^n)$ is of quasi-polynomial type,  say of period
$g$ and degree $c$. We write 
\[
f^I_J(n) = a_c(n)n^c + a_{c-1}(n)n^{c-1} + \text{lower terms}.
\]
In \cite[ 2.4]{HPV} we proved that  $\dim I_n(J)/I^n$ is constant for $n \gg 0$ and $a_c(-)$ is a  positive constant.  In this short paper we prove
\begin{theorem}\label{main}[with hypotheses as in \ref{int}] Further assume that $I$ is generated by some elements of the same degree and $\height I \geq 2$. Then
$a_{c-1}(-)$ is also a constant.
\end{theorem}
\section{Preliminaries}
Let $K$ be a field and let $R = \bigoplus_{m \geq 0}R_m$ be a standard graded, finitely generated $K = R_0$-algebra. Let $M = \bigoplus_{m \geq 0}M_m $ be a graded $R$-module. Its Hilbert series is $H_M(z) = \sum_{m \geq 0} \dim_K M_m$. It is well-known that $H_M(z) = h_M(z)/(1-z)^{\dim M}$ where $h_M(z) \in \Z[z]$. The number $e_0(M) = h_M(1)$ is called the \emph{multiplicity} of $M$. 

\s Let $I$ be a homogeneous ideal in $R$. An element $x \in I$ is called $I$-\emph{superficial} if there exists $a$ and $n_0$ such that 
$(I^{n+1} \colon x)\cap I^a = I^n$ for all $n \geq n_0$. Superficial elements exist when $K$ is infinite. If $\grade(I) > 0$ then by using Artin-Rees lemma it can be shown that $(I^{n+1} \colon x)  = I^n$ for all $n \gg 0$. If $I = (u_1, \ldots, u_l)$ and $K$ is infinite then a general linear sum of the $u_i^{s}$ is an $I$-superficial element. Thus $I$-superficial    elements need not be homogeneous even if 
$I$ is. However if $I$ is generated by some elements of the same degree then it is clear that there exists homogeneous  $I$-superficial elements.

The following result can be shown in the same way as in the proof of  Proposition 2.3 and Theorem 2.4 in \cite{HPV}.
\begin{theorem}\label{gen}(with setup as above). Let $\F =\{ J_n \}_{n \geq 0}$ be a multiplicative filtration of homogeneous ideals with $J_0 = R$ and $I \subseteq J_1$. Further assume that  $\Sc(\F) = \bigoplus_{n \geq 0}J_n$ is a finitely generated  $R$-algebra. Then the function
$f^I_\F(n) = e_0(J_n/I^n)$ is of quasi-polynomial type, say of degree $c$ and period $g$. For $n \gg 0$
we write $f^I_\F(n) = a_c(n)n^c + \text{lower terms}$. If $\grade(I) > 0$ then  $\dim (J_n/I^n)$ is constant for $n \gg 0$ and $a_c(-)$ is a constant.
\end{theorem}

\s \label{L}  Let $\F =\{ J_n \}_{n \geq 0}$ be a multiplicative filtration of homogeneous ideals with $J_0 = R$ and $I \subseteq J_1$.
 Let $\R = R[It]$ be the Rees-algebra of $I$. Then $L^\F = \bigoplus_{n \geq 0} R/J_{n+1}$ has a structure of $\R$-module. This can be seen as follows.
Note $\Sc(\F) = \bigoplus_{n \geq 0}J_n$ is a graded $\R$-module. We have a short exact sequence 
\[
0 \rt \Sc(\F) \rt R[t] \rt L^\F(-1) \rt 0.
\]
This gives an $\R$-module structure on $L^\F(-1)$ and so on $L^\F$.
\section{Proof of Theorem \ref{main}}
In this section we give
\begin{proof}[Proof of Theorem \ref{main}]
We may assume that $K$ is an infinite field.  \\
Let $f(n) = e_0(I_n(J))/I_n)$. Then by our earlier result $f$ is of polynomial type say of degree $c$ and period $g$.  Say
\[
f(n) = a_c(n)n^c + a_{c-1}(n)n^{c-1} + \text{lower terms} \quad \text{for } n \gg 0,
\]
where $a_c(n) = a$ is a positive constant. We have nothing to show if $c = 0$. So assume $c > 0$. We also have $\dim I_n(J)/I^n$  is constant for $n \gg 0$. Let this constant be $r$.

As $I$ is generated by some elements of the same degree we can choose homogeneous $u$ which is $I$-superficial. Say $(I^{n+1} \colon u) = I^n$ for $n \geq n_0$.

Claim-1:  $(I_{n+1}(J) \colon u) = I_{n}(J)$ for $n \geq n_0$. \\
Fix $n \geq n_0$. Let $p \in (I_{n+1}(J) \colon u) $. Then $up \in I_{n+1}(J)$. So $upJ^r \subseteq I^{n+1}$ for some $r > 0$. Then
$pJ^r \subseteq (I^{n+1}\colon u) = I^n$. So $p \in I_n(J)$.

Set $\F = \{I_n(J) \}_{n \geq 0}$.  Let $L^\F = \bigoplus_{n \geq 0 }A/I_{n+1}(J)$ and $L^I = \bigoplus_{n \geq 0 }A/I^{n+1}$ be given the $\R = A[It]$-module structures as described in \ref{L}. Set \\  $W = \bigoplus_{n \geq 0 }I_{n+1}(J)/I^{n+1}$. Then we have a short exact sequence of $\R$-modules
\[
0 \rt W \rt L^I \rt L^\F \rt 0. \tag{*}
\]
Let $p = ut$. By claim-1;  $\ker (L^\F(-1) \xrightarrow{p} L^\F)$ is concentrated in degrees $\leq n_0$. So we have a short exact sequence
\[
0 \rt \ov{W}_{n \geq n_0 + 1} \rt \ov{L^I}_{n \geq n_0 + 1} \rt  \ov{L^\F}_{n \geq n_0 + 1} \rt 0; \tag{**}
\]
(here $\ov{(-)} = (-)/p(-)$). We note that $\ov{L^\F}_n = A/(I_{n+1}(J), u)$ and \\  $\ov{L^I}_n = A/(I^{n+1}, u)$. Consider the standard graded $K$-algebra $R = A/(u)$. Consider the $R$-ideal $\ov{I} = I/(u)$ and the Noetherian $R$-filtration $\ov{\F } = \{ \ov{I_n(J)} \}_{n \geq 0} $ where $\ov{I_n(J)} = (I_n(J) + (u))/(u)$. Notice $\grade I = \height I \geq 2$. So $\grade \ov{I} \geq 1$. Set $g(n) = e_0(\ov{I_n(J)}/\ov{I^n})$. By \ref{gen} we have $\dim \ov{I_n(J)}/\ov{I^n}$ is constant for $n \gg 0$ (say equal to $s$). Furthermore $g(n)$ is of quasi-polynomial type say of period $g'$ and degree $l$. Say 
\[
g(n) = b_l(n)n^l + \text{lower terms}.
\]
We also have $b_l (-) = b$ a constant.

By (*) and (**) we have a short exact sequence for all $n \gg 0$
\[
 0 \rt I_{n-1}(J)/I^{n-1}  \rt I_n(J)/I^n  \rt \ov{I_n(J)}/\ov{I^n} \rt 0.
\]
So we have  $r \geq s$. If $r > s$ then notice $c = 0$ which is a contradiction. So $r = s$. It follows that $l = c -1$. As $f(n) - f(n-1) =g(n)$
we get that 
\[
ac + a_{c-1}(n) - a_{c-1}(n-1)  = b  \quad \text{for all} \ n \gg 0. 
\]
So for $n \gg 0$ we have  $a_{c-1}(n) = \alpha n + \beta$ for  $n \gg 0 $ for some constants $\alpha, \beta$; see \cite[4.1.2]{BH}.
As $a_{c-1}(-)$ is periodic it follows from  that $\alpha = 0$ and so $a_{c-1}(-)$ is a constant. The result follows.
\end{proof}


\begin{thebibliography} {99}
\bibitem {BI} W. Bruns and B. Ichim,
\emph{On the coefficients of Hilbert quasipolynomials},
Proc. Amer. Math. Soc. 135 (2007) no. 5, 1305-1308.

\bibitem {BH}  W. Bruns and J. Herzog,
Cohen-Macaulay Rings, revised edition,
Cambridge Studies in Advanced Mathematics, 39.
Cambridge University Press, 1998.


\bibitem {HHT}
J. Herzog, T. Hibi and N. V. Trung,
\emph{Symbolic powers of monomial ideals and vertex cover algebras}, Advances in Math. 210 (2007), 304-322.

\bibitem{HPV}
J.~Herzog, T.~J.~Puthenpurakal and J.~K.~Verma, 
{\em Hilbert polynomials and powers of ideals}, 
Math. Proc. Cambridge Philos. Soc. 145 (2008), no. 3, 623--642. 

\bibitem{M} P. McMullen. {\em Lattice invariant valuations on rational
polytopes}, Arch. Math.  31 (1978/79), 509--516.


\bibitem{S} R. Stanley. {\em Decompositions of rational convex
polytopes}, Ann. Discr. Math.  6, (1980), 333--342.
\end{thebibliography}
\end{document}